\documentclass[12pt]{article}
\usepackage{amsmath,amsfonts,amssymb,amsthm,amscd}
\usepackage{xcolor}
\title{Galois theory, automorphism groups of prime models, and the Picard-Vessiot closure}
\date{\today}
\author{David Meretzky \\{Universidad de Los Andes}\and Anand Pillay\thanks{Supported by NSF grants  DMS-2054271 and DMS-2502292}\\{University of Notre Dame}}

\newtheorem{Theorem}{Theorem}[section]
\newtheorem{Proposition}[Theorem]{Proposition}
\newtheorem{Definition}[Theorem]{Definition}
\newtheorem{Remark}[Theorem]{Remark}
\newtheorem{Lemma}[Theorem]{Lemma}
\newtheorem{Corollary}[Theorem]{Corollary}

\newcommand{\C}{\mathbb C}

\begin{document}
\maketitle

\begin{abstract}  We work in the context of a complete totally transcendental theory $T = T^{eq}$.  We consider the prime model  $M_{A}$ over a set $A$.  For intermediate sets $B$ with $A\subseteq B \subseteq M_{A}$ which are normal ($Aut(M_{A}/A)$-invariant) and ``minimal" we give a full Galois correspondence between intermediate definably closed sets  $A\subseteq B \subseteq M_{A}$ and 
``closed" subgroups of $Aut(B/A)$ (the group of $A$-elementary permutations of $B$).  

The unique greatest such minimal normal $B$ coincides with Poizat's ``minimal closure" $A_{min}$, so our paper can be see as extending the well-known Galois correspondence between closed subgroups of the profinite group $Aut(acl(A)/A)$ and intermediate definably closed sets, to the case of $A_{min}$ in place of $acl(A)$.

The main result applies  in particular to the ``Picard-Vessiot closure"  $K^{PV_{\infty}}$ (or $K_{\infty}$)  of a differential field $K$ of characteristic $0$ with algebraically closed field $C_{K}$ of constants. We also show that normal differential subfields of $K^{PV_{\infty}}$ containing $K$ are ``iterated $PV$-extensions"  of $K$, and the Galois correspondence result above holds for these extensions.  This fills in  some missing parts of Magid's paper \cite{MagidII}. 

 We also discuss exact  sequences  $1 \to N \to G \to H \to 1$, where  $G = Aut(K_{2}/K)$, $N = Aut(K_{2}/K_{1})$ and $H = Aut(K_{1}/K)$, $K_{1}$ is a (maybe infinite type) $PV$ extension of $K$. $K_{2}$ is a (maybe infinite type) $PV$ extension of $K_{1}$ and $K_{2}$ is normal over $K$ (in the differential closure of $K$) and again $C_{K}$ is algebraically closed.   Both $N$ and $H$ have the structure of proalgebraic groups over $C_{K}$.  We show that conjugation by any given element of $G$ is a proalgebraic automorphism of $N$. Moreover   if  $G$ splits as a semidirect product  $N\rtimes H$, then left multiplication by any fixed  element of $G$ is a morphism of proalgebraic varieties $N\times H \to N\times H$. This improves and extends observations in Section 4 of \cite{MagidII} which dealt with one particular example.

\end{abstract}


\section{Introduction}
This paper extends and builds on the paper \cite{Pillay-aut} which introduced (relatively) definable subsets of automorphism groups of prime models, proved the existence and uniqueness of invariant measures
on these automorphism groups, and related these notions to differential Galois theory and definable Galois cohomology. In particular the notion was used to characterize definable cocycles (first defined in  \cite{Pillay-Galois}, and was extended in   \cite{Meretzky-definable-GC}).

The current paper is closely related to work of Poizat (see \cite{Poizat-Galois}, \cite{Poizat-theories-stables} and Section 2, Chapter 18 of \cite{Poizat-course}) and can even be thought of as a continuation of it. Sections 3 and 4 of the current paper are also closely related to the unpublished paper \cite{MagidII} of Andy Magid, which in several senses we complete and extend. 

We are concerned here with a full Galois correspondence.   In \cite{Pillay-aut}  we worked mainly in the context of a stable theory $T$, and $M$ a model of $T$ which is atomic and strongly $\omega$-homogeneous over a subset $A$. (Strongly $\omega$-homogeneous over $A$ means that whenever ${\bar a}$, ${\bar b}$ are finite tuples in $M$ with the same type over $A$ then there is an automorphism $f$ of $M$ fixing $A$ pointwise and taking ${\bar a}$ to ${\bar b}$.)  In this paper it is convenient to assume that $T$ is totally transcendental (called $\omega$-stable when $T$ is countable.)  This guarantees the existence and uniqueness of prime models over arbitrary subsets which in particular have the properties above (atomicity and strong $\omega$-homogeneity) and others too.  So from now on $T$ will denote a complete, totally transcendental theory. 

We will be working freely in $T^{eq}$ and let ${\bar M}$ denote a  very saturated model (which exists as $T$ is $t.t$).  $A, B$ will typically denote small definably closed subsets of ${\bar M}$. For $A\subseteq B$, $Aut(B/A)$ denotes the group of $A$-elementary permutations of $B$. 

We let $M_{A}$ denote the prime model over $A$ (which is unique up to $A$ isomorphism). Before stating the results we give a quick account of the main notions, which will be elaborated on later. 
Let $A\subseteq B \subseteq M_{A}$. We call $B$ invariant or normal  (over $A$ in $M$) if $B$ is fixed setwise by $Aut(M_{A}/A)$. We call $B$ minimal (over $A$ in $M$) if there is no proper subset $C$ of $B$ with $A\subseteq C$ and $C$ Tarski-Vaught in $B$.  Let ${\bar b}$ be an enumeration of $B$, $p_{0}({\bar x}) = tp({\bar b}/A)$ and $S_{p_{0}}(B)$ the (Stone) space of extensions of $p_{0}$ to complete types $q({\bar x})$ over $B$.  

For $X$ a closed subset of $S_{p_{0}}(B)$ given by a set $\Sigma({\bar x}, {\bar b})$ of formulas, by $X(Aut(B/A))$ we mean 
$\{\sigma\in Aut(B/A): \models \Sigma(\sigma({\bar b}), {\bar b})\}$, and  $X(Aut({\bar M}/A)) =\{\sigma\in Aut({\bar M}/A): \models \Sigma(\sigma({\bar b}), {\bar b})\}$. 
When $X$ is clopen, then $X$ is determined by $X(Aut(B/A))$ as pointed out in \cite{Pillay-aut}, but not necessarily when $X$ is arbitrary closed. On the other hand $X$ is determined by $X(Aut({\bar M}/A))$. By a {\em closed subgroup} of $Aut(B/A)$ we mean a subgroup of the form $X(Aut(B/A))$ where $X$ is a closed subset of $S_{p_{0}}(B)$ and $X(Aut({\bar M}/A))$ is a subgroup of $Aut({\bar M}/A)$. 

The main result proved in Section 2 is: 
\begin{Theorem} Suppose that $A\subseteq B \subseteq M_{A}$ where $B$ is normal (over $A$ in $M_{A})$ and minimal.  Then there is a Galois correspondence between closed subgroups of $Aut(B/A)$ and sets $C$ with $A\subseteq C \subseteq B$. 
\end{Theorem}
 When $B = acl(A)$ (without the need for any $t.t.$ assumption on $T$), this is Theorem 14 of \cite{Poizat-Galois}, after noting as in \cite{Pillay-aut} that closed subgroups of $Aut(acl(A)/A)$ in the sense of this paper coincide with closed subgroups of $Aut(acl(A)/A)$ as a profinite group. 
Also at the end of Section 2, we make the connection with Poizat's minimal closure from \cite{Poizat-theories-stables}.

\vspace{5mm}
\noindent
We now discuss the example of the theory $DCF_{0}$ of differentially closed fields of characteristic zero and Picard-Vessiot extensions of a differential subfield.  $DCF_{0}$ is $\omega$-stable so the paragraphs above apply.  The prime 
model over a differential subfield $K$ is called the differential closure of $K$ and often denoted $K^{diff}$.  The classical notion of a Picard-Vessiot extension of $K$ (assuming the field $C_{K}$ of constants is algebraically closed)  is a differential field extension of $K$ generated over $K$ by a 
``fundamental  system of solutions" in $K^{diff}$ of a linear differential equation $\partial Y = AY$ over $K$, namely $Y$ is a column  vector of unknowns and $A$ an $n\times n$ matrix over $K$.  It is convenient to follow the 
conventions of Wibmer  in  for example \cite{Wibmer-regular-singular}, where this is referred to as a Picard-Vessiot extension for the equation $\partial Y = AY$ . More generally Wibmer defines a Picard-Vessiot extension of $K$ for a 
(possibly infinite) family ${\mathcal F}$ of linear differential equations over $K$ as the differential field generated over $K$ by the union of all fundamental systems of solutions  (in $K^{diff}$) of all equations in ${\mathcal F}$. He calls a 
Picard-Vessiot extension ``of finite type" if it is the Picard-Vessiot extension for a single linear differential equation (equivalently finitely many).  With these conventions there is a greatest Picard-Vessiot extension of $K$ (inside a given copy 
of $K^{diff}$), which we call $K^{PV}$.  Then $K^{PV}$ will be normal and minimal over $K$ and $Aut(K^{PV}/K)$ has the structure of a proalgebraic group (with respect to $C_{K}$). The closed subgroups in our case correspond to Zariski  closed subgroups and Theorem 1.1 gives the well-known Galois correspondence.  Actually $K^{PV}$ also has the structure of a prodefinable or pro-differential algebraic group defined over $K$, which is isomorphic over $K^{PV}$ to the corresponding proalgebraic over $C_{K}$ group.

In any case Theorem 1.1 covers other cases where it gives new results.
One such is the ``Picard-Vessiot closure" of a differential field which will be looked at in Section 3.  Given again a differential field $K$ with algebraically closed field of constants, we have as 
above $K^{PV}$. In contradistiction with the case of $K^{alg}$ the algebraic closure of $K$, we typically have that $(K^{PV})^{PV}$ properly contains $K^{PV}$.  For example when $K = 
\C$, the complex field considered as a field of constants, then $Aut(K^{PV}/K)$ is the connected component of the proalgebraic group on one generator.  On the other hand $
\C(t)$ (with $d/dt$) is contained in $K^{PV}$ and  $Aut(\C(t)^{PV}/\C(t))$ is the free proalgebraic group on continuum many generators (and of course is contained in 
$Aut((K^{PV})
^{PV}/K^{PV})$).  We define  $K^{PV_{0}} =   K$,  $K^{PV_{n+1}} = (K^{PV_{n} })^{PV}$, and $K^{PV_{\infty}} = \cup_{n}K^{PV_{n}}$. 

Already in \cite{Poizat-Galois} and \cite{Poizat-theories-stables} where the minimal closure is introduced, it is stated that $K^{PV}$ is normal and minimal being contained in the minimal closure $K_{min}$ of $K$ (see also Lemma 2.14 of the current paper).  But more is true. 
    
\begin{Proposition}
    Each of $K^{PV_n}$ and $K^{PV_\infty}$ are normal and minimal over $K$ in $K^{diff}$, (so by Lemma 2.14 contained in $K_{min}$). 
\end{Proposition}

Hence Theorem 1.1 applies to give the Galois correspondence.  Magid calls $K^{PV_{\infty}}$ the complete PV closure of $K$ (although $PV$ closure would be correct and enough). In \cite{MagidII} a full Galois correspondence for 
$Aut(K^{PV_{\infty}}/K)$ is given; between intermediate differential fields and ``closed" subgroups. However no definition of closed subgroup is offered (other than the tautological definition of being precisely the subgroups of the 
form $Fix(L)$ for $L$ an intermediate differential field).  So Theorem 1.1  fills this gap in a reasonably canonical form. Neverthess there is some content in \cite{MagidII}, namely that the set of fixed points of $Fix(L)$ is precisely $L$. The 
same observation, in a more general model-theoretic environment, is made in \cite{Poizat-Galois} ( top of p. 1166). 

Actually Magid's Galois correspondence (Theorem 5 of \cite{MagidII}) is for what he calls normal ``locally iterated" $PV$ extensions.   By our conventions that $PV$ extensions may be of infinite type, we could and will rename these ``iterated 
$PV$ extensions". Such an object is a differential field $L$ with $K\leq L \leq K^{PV_{\infty}}$ such that EITHER  $L$ is the union of a strictly increasing chain $(L_{i}: i < \omega)$, with $L_{0} = K$ and each $L_{i+1}$ a $PV$ extension of 
$L_{i}$, OR there is a finite such chain $(L_{i}:i<n)$ with $L = L_{n}$.   In particular each $K^{PV_{n}}$ as well as $K^{PV_{\infty}}$ are examples of iterated $PV$ extensions.

  We will  prove:
\begin{Proposition}  Let $K\leq L \leq K^{PV_{\infty}}$ be normal over $K$. Then $L$ is an iterated $PV$ extension of $K$ where moreover each $L_{i}$ is also normal and minimal (over $K$ in $K^{diff}$).
\end{Proposition} 

Hence Theorem 1.1 applies to any normal  $L\leq K^{PV_{\infty}}$ containing $K$. 

\vspace{2mm}
\noindent
Now for $n>1$, $Aut(K^{PV_{n}}/K)$ can not have any structure as a  prodefinable group in $DCF_{0}$ or equivalently as a proalgebraic group with respect to the constants.  
So one can ask what additional information (other than having a composition series of proalgebraic groups in the constants) does one have. 

We will do an analysis of the case where $n=2$.  In fact we work at the more general level of $2$-iterated $PV$-extensions.  Let $K_{1}$ be a PV extension of $K$, and let $K_{2}$ be a PV extension of $K_{1}$ such that $K_{2}$ is normal in $K^{diff}$ over $K$.  Both $K_{1}$ and $K_{2}$ may be of infinite type. A special case would be $K_{1} = K^{PV_{1}}$ and $K_{2} = K^{PV_{2}}$.

We let $G = Aut(K_{2}/K)$, $N= Aut(K_{2}/K_{1})$, and $H = Aut(K_{1}/K)$.  Then $N$ and $H$ have the structure of  proalgebraic groups with respect to the algebraically closed field $C_{K}$, and we have the exact sequence 
$$1\to N\xrightarrow{i} G \xrightarrow{\pi} H \to 1$$ where $i$ is inclusion.


\begin{Proposition}\label{prop: proalg}
(i) Let $\tau\in G$. Then conjugation by $\tau$ is a proalgebraic automorphism of $N$.
\newline
(ii) Suppose that $G$ splits as the semidirect product of $N$ and $H$ then left translation by a given element of $G$ is a proalgebraic morphism from 
the underlying proalgebraic variety of $G$, $N \times H$, to itself. 
\end{Proposition} 

The proof of Proposition \ref{prop: proalg} will go through the  the general structure theory developed in \cite{Kamensky-Pillay}.

\vspace{5mm}
\noindent
In Section 4 of \cite{MagidII} Magid studies a particular example. The base field $K$ is ${\mathbb C}(t)$.  $F_{1}$ is a certain finite type Picard-Vessiot extension of $K$ with Galois group  $H$ equal to the additive group of the constants. $E$ is a certain  (infinite type)  Picard-Vessiot extension of $F_{1}$ which is normal (in $K^{diff}$ over $K$) and whose Galois group $N$  is the direct product of continuum many copies of the additive group of the constants, indexed by the constants, a proalgebraic group.  Magid observes that if $G = Aut(E/K)$ then $G$ is the semidirect product of $N$ and $H$.  In this special case Magid computes   that left multiplication  by any given element of $G$ is a proalgebraic morphism of the proalgebraic variety $N\times H$ to itself  (but not right multiplication).  Of course this is just a special case of our Proposition \ref{prop: proalg} (ii).

In passing,  we know of no (nontrivial) example where $Aut(K^{PV_{2}}/K)$ is a semidirect product of $Aut(K^{PV_{2}}/K^{PV_{1}})$ and $Aut(K^{PV_{1}}/K)$.

\vspace{2mm}
\noindent
{\em Acknowledgements.} The second author would like to thank Andy Magid and Bruno Poizat for discussions and questions, especially concerning the automorphism group of $K^{PV_{2}}$ over $K$. 

\section{The Galois theory}
We will here prove Theorem 1, after expanding on the notions introduced in the introduction above. 

We take $T$ to be a complete $t.t.$ theory in language $L$.  When $L$ is countable this means precisely that $T$ is $\omega$-stable.   As in Section 1, we work in a saturated monster model ${\bar M}^{eq}$, $A$ and $B$ denote small definably closed subsets, and $M_{A}$ denotes (a copy of) the prime model over $A$. 
We fix $B$ with $A\subseteq B \subseteq M_{A}$. We defined $B$ to be {\em normal} (in $M_{A}$ over $A$) if $B$ is setwise invariant under $Aut(M_{A}/A)$. This just means that for any finite tuple $b$ from $B$, all realizations of $tp(b/A)$ in $M_{A}$ are already in $B$.  By $Aut(B/A)$ we denote the group of elementary (in the sense of $M_{A}$ or ${\bar M}$) over $A$ permutations of $B$. By our assumptions this is precisely the set of restrictions to $B$ of elements of $Aut(M_{A}/A)$.
In \cite{Pillay-aut} we introduced ``definable" subsets of $Aut(B/A)$, which we  will now call {\em relatively definable}  subsets (reverting to the original language from \cite{KPR}):  For a formula $\phi({\bar x}, {\bar y})$ over $A$, with ${\bar x}$, ${\bar y}$ finite tuples of variables of the same length, and a tuple ${\bar b}$ from $B$ of the same length, 
$\{\sigma\in Aut(B/A): \models \phi(\sigma({\bar b}), {\bar b})\}$ is a relatively definable subset of $Aut(B/A)$. 

We now let ${\bar b}$ be an enumeration  of $B$, and $p_{0}({\bar x}) = tp({\bar b}/A)$. So now ${\bar x}$ is a possibly infinite tuple of variables. We let $S_{p_{0}}(B)$ be the (Stone) space of complete types $p({\bar x})$ over $B$ which extend $p_{0}$.  In \cite{Pillay-aut} we pointed out that relatively definable subsets of $Aut(B/A)$ are in natural one-one-correspondence with clopen subsets of $S_{p_{0}}(B)$.  (We will  repeat this below).  It is natural to ask: What corresponds to arbitrary closed subsets of $S_{p_{0}}(B)$? 

\begin{Definition} Let $X$ be a closed subset of $S_{p_{0}}(B)$ given by the set $\Sigma({\bar x}, {\bar b})$ of formulas, where $\Sigma({\bar x}, {\bar y})$ is a collection of $L_{A}$-formulas and ${\bar b}$ an enumeration of $B$. 
\newline
(i) $X(Aut({\bar M}/A)) =\{ \sigma\in  Aut({\bar M}/A):   {\bar M}\models \Sigma(\sigma({\bar b}), {\bar b}) \}$, and
\newline
(ii) $X(Aut(B/A)) =  X(Aut({\bar M}/A))|_{B}$  $(= \{\sigma\in Aut(B/A): \models \Sigma(\sigma({\bar b}), {\bar b})\})$. 
\end{Definition}

\begin{Lemma} (i) Let $X$ be a closed subset of $S_{p_{0}}(B)$. Then $X$ is determined by $X(Aut({\bar M}/A))$.
\newline
(ii) Let $X$ be a clopen subset of $S_{p_{0}}(B)$. Then $X$ is determined by  $X(Aut(B/A))$.
\end{Lemma}
\begin{proof} (i)  Suppose $X,Y$ are closed subsets of $S_{p_{0}}(B)$ given by $\Sigma({\bar x}, {\bar b})$ and $\Gamma({\bar x}, {\bar b})$ respectively.
Suppose $X(Aut({\bar M }/A)) = Y(Aut({\bar M}/A))$.  Let ${\bar c}$ realize $\Sigma({\bar x},{\bar b})$. As ${\bar c}$ realizes $p_{0}$, there is $\sigma\in Aut({\bar M}/A)$ with 
$\sigma({\bar b}) = {\bar c}$. So $\sigma\in X(Aut({\bar M}/A))$, whereby $\sigma\in Y(Aut({\bar M}/A))$, whereby ${\bar c}$ realizes $\Gamma({\bar x}, {\bar b})$. This  suffices.
\newline
(ii) This is essentially Proposition 2.8 (ii) of \cite{Pillay-aut}. But as our notation here is slightly different, we give the proof.   We let $x,y$ denote finite tuples of variables of an appropriate length. Let $b_{0}$ be a finite tuple from $B$ and let $\phi(x,y)$, $\psi(x,y)$ be formulas over $A$, such that for $\sigma\in Aut(B/A)$, $\models \phi(\sigma(b_{0}), b_{0})$ iff $\models \psi(\sigma(b_{0}), b_{0})$. We want to prove that $\phi(x,b_{0})$ is equivalent to $\psi(x,b_{0})$ in $S_{p_{0}}(B)$.  Let $\chi(y)$ isolate $tp(b_{0}/A)$.  We may assume that each of $\phi(x,b_{0})$, $\psi(x,b_{0})$ implies $\chi(x)$. 
We will show that $\phi(x,b_{0})$ and $\psi(x,b_{0})$ are equivalent. For this it suffices to work in the model $M_{A}$. If the formulas are not equivalent then without loss there is $c_{0}\in M_{A}$ such that $\models \phi(c_{0},b_{0})\wedge \neg\psi(c_{0},b_{0})$. As $tp(c_{0}/A) = tp(b_{0}/A)$, there is $\sigma\in Aut(M_{A}/A)$ such that $\sigma(b_{0}) = c_{0}$. 
Now $c_{0}\in B$ by normality and $\tau = \sigma|B \in Aut(B/A)$. So $\models \phi(\tau(b_{0}), b_{0})$ but $\models \neg\psi(\tau(b_{0}), b_{0})$, a contradiction. 
\end{proof}

\begin{Remark} Clearly in general an arbitrary closed subset $X$ of $S_{p_{0}}(B)$ need not be determined by $X(Aut(B/A))$. For example suppose $b_{i}\in B$ is not in $acl(A)$ and let $\Sigma({\bar x}, {\bar b})$ express that $x_{i}\neq b_{j}$ for all $j$. Then  $\Sigma({\bar x}, {\bar b})$ determines a nonempty closed subset of $S_{p_{0}}(B)$,  but there is no $\sigma\in Aut(B/A)$ such that $\models \Sigma(\sigma({\bar b}), {\bar b}))$. 
\end{Remark}  

We get a Stone topology on $Aut({\bar M}/A)$ induced from $S_{p_{0}}(B)$ and we call the corresponding subsets of $Aut({\bar M}/A)$, $B$-closed. 
It then makes sense to talk about a $B$-closed subgroup of $Aut({\bar M}/A)$. 


\begin{Definition} By a closed subgroup $G$ of $Aut(B/A)$ we mean $H|_B$ for some $B$-closed subgroup $H$ of  $Aut({\bar M}/A)$.  We will say that $\Sigma({\bar x}, {\bar b})$ defines $G$ if it defines $H$ in $Aut({\bar M}/A)$. 
\end{Definition} 

Note that a closed subgroup of $Aut(B/A)$ is in particular a subgroup of $Aut(B/A)$ which is of the form $X(Aut(B/A))$ for some closed subset of $S_{p_{0}}(B)$, but this is NOT an equivalent description. 

The next Lemma holds with no conditions on $T$, only that  $A\subseteq B$ are small subsets of ${\bar M}$ (and Definition 2.1 (i) makes sense at this general level). 

\begin{Lemma}  $\Sigma({\bar x}, {\bar b})$ defines a $B$-closed subgroup $H$ of $Aut({\bar M}/A)$ if and only if   $\Sigma({\bar x}, {\bar y})$ defines an equivalence relation on realizations of $p_{0}$ (in ${\bar M}$).

\end{Lemma}
\begin{proof} This is the same as the proof of  Lemma 3.1 of \cite{Pillay-aut} which is about the case when $X$ is clopen (so given by a single formula)  and $B = M_{A}$.  We will give it anyway. First prove left implies right.  So we assume $\Sigma({\bar x}, {\bar b})$ defines a $B$-closed subgroup $H$ of $Aut({\bar M}/A)$
\newline
(i) Reflexive: As the identity is in $H$, $\Sigma({\bar b}, {\bar b})$ holds whereby $\Sigma({\bar x}, {\bar x})\subseteq p_{0}({\bar x})$.
\newline
(ii) Symmetric:  Suppose  that ${\bar c}$ realizes $p_{0}$ in $\bar M$ and $\models \Sigma({\bar c}, {\bar b})$.  Then there is $\sigma\in H$ such that $\sigma({\bar b}) = {\bar c}$.  But then $\sigma^{-1}\in H$, whereby $\models \Sigma(\sigma^{-1}({\bar b}), {\bar b})$. As $\Sigma({\bar x}, {\bar y})$ is over $A$ we can apply $\sigma$ to get 
$\models \Sigma({\bar b}, \sigma({\bar b})$, namely $\models \Sigma({\bar b}, {\bar c})$. 
\newline
We have shown that $p_{0}({\bar x})$ and $\Sigma({\bar x}, {\bar b})$ implies $\Sigma({\bar b}, {\bar x})$.  Hence $\Sigma$ is symmetric on all realizations of $p_{0}$.
\newline
(iii) Transitive:  Given that we have symmetry, it suffices to prove that if ${\bar c}$, ${\bar d}$ realize $p_{0}$ and $\models\Sigma({\bar c}, {\bar b})$ and $
\models\Sigma({\bar d}, {\bar b})$, then $\models\Sigma({\bar d}, {\bar c})$.  So we have $\sigma, \tau\in H$ such that $\sigma({\bar b}) = {\bar c})$ and $\tau({\bar b}) 
= {\bar d}$. Then $\sigma^{-1}\tau\in H$, whereby $\models \Sigma(\sigma^{-1}\tau({\bar b}), {\bar b})$, so $\models\Sigma(\sigma^{-1}({\bar d}), {\bar b})$. Apply 
$\sigma$ to get $\Sigma({\bar d}, {\bar c})$, as required. 
\newline
The right implies left direction just reverses the proofs above. For example if $\Sigma({\bar x}, {\bar x})$ holds for ${\bar x}$ realizing $p_{0}$ then $\Sigma({\bar b}, {\bar b})$ holds, so the identity is in $H$. 
\end{proof}

We now bring in  (some of) the assumptions.

\begin{Lemma}  $\Sigma({\bar x}, {\bar y})$ be as in Lemma 2.5. Then there are $A$-definable equivalence relations  $\phi_{i}({\bar x}, {\bar y})$ such that on realizations of $p_{0}$,  $\Sigma({\bar x}, {\bar y})$ agrees with $\{\phi_{i}({\bar x}, {\bar y}): i\in I\}$. 
\end{Lemma} 
\begin{proof} This is precisely elimination of hyperimaginaries in stable theories.  See for example \cite{Pillay-Poizat} or \cite{Lascar-Pillay}.
\end{proof}

\begin{Proposition}  (i) The $B$-closed subgroups of $Aut({\bar M}/A)$ are precisely subgroups of the form  $\{\sigma\in Aut({\bar M}/A): \sigma$ fixes $C$ pointwise\} for some $C\subseteq B$.
\newline
(i) Likewise the  closed subgroups of $Aut(B/A)$ are precisely subgroups of the form $\{\sigma\in Aut(B/A): \sigma$ fixes $C$ pointwise\} for some $C\subseteq B$. 
\end{Proposition} 
\begin{proof} Remember we are working in ${\bar M}^{eq}$.  First fix a subset $C$ of $B$. Suppose $c\in C$ and suppose $c = b_{\alpha}$ with respect to the given enumeration ${\bar b}$ of $B$. Then the formula $x_{\alpha} = b_{\alpha}$ defines $\{\sigma\in Aut({\bar M}/A): \sigma(c) = c\}$.  So the collection of all these formulas, as $c$ ranges over $C$, defines the $B$-closed subgroup of $Aut({\bar M}/A)$ which is precisely the set of $\sigma$ which fix $C$ pointwise. 

Conversely, let $H$  be a closed subgroup of $Aut({\bar M}/A)$. By Lemma 2.7  it is defined by a collection of formulas $\phi_{i}({\bar x}, {\bar b})$  ($i\in I$) where each $
\phi({\bar x}, {\bar y})$ is a formula over $A$ defining an equivalence relation on realizations of $p_{0}$. Fix $i$. Then only finitely many variables, actually appear in $
\phi_{i}$ which we  assume for simplicity to be of the form $(x_{1},..,x_{n}, y_{1},..,y_{n})$ (where $y_{1},..,y_{n}$ correspond to $b_{1},..,b_{n}$).  Then $\phi_{i}$ 
defines an equivalence relation $E_{i}$ on realizations of $tp(b_{1},..,b_{n}/A)$, and $\phi_{1}(x_{1},..,x_{n}, b_{1},..,b_{n})$ defines $Fix(b_{1},..,b_{n}/E_{i})$, the 
automorphisms which fix this imaginary $c_{i} = (b_{1},..,b_{n})/E_{i}$. But $B$ is definably closed in ${\bar M}^{eq}$, so $c_{i}\in B$. Let $C\subseteq B$ be the set of 
these $c_{i}$. 
\newline
(ii)  follows as  the closed subgroups of $Aut(B/A)$ are by definition the intersections of $B$-closed subgroups of $Aut({\bar M}/A)$ with $Aut(B/A)$.
\end{proof} 

\begin{Remark} (i)   Proposition 2.7  is valid for any definably closed $B$ (containing $A$) under the  additional elimination of hyperimaginaries assumption.   Namely there is no prime model $M_{A}$ in the picture.
\newline
(ii) From Proposition 2.7 we get a tautological Galois correspondence between $B$-closed subgroups of $Aut({\bar M}/A)$ and definably closed subsets of $B$ containing $A$.  All that has to be shown is that if $C$ is a subset of $B$ containing $A$, then the set of elements of $B$ (even of ${\bar M}$) fixed by the group of automorphisms of $\bar M$ which fix $C$ pointwise is $dcl(C)$, which is immediate by saturation and homogeneity of $\bar M$. 
\newline
(iii) In the context of $T$ being $t.t.$, and $B$ normal over $A$ in $M_{A}$,   obtaining a full Galois correspondence between closed subgroups of $Aut(B/A)$ and definably closed subsets of $B$ containing $A$, depends again on showing that the elements of $B$ fixed by the set (group) of $\sigma\in Aut(B/A)$ which fix some $C\subseteq B$ is precisely $dcl(A,C)$.  This requires an additional assumption on $B$ which we introduce below. 
\end{Remark}

Recall the basic notion of ``Tarski-Vaught" (T-V) subset:  Let  $C\subseteq D$ be sets of parameters in $\bar M$.  We will say that $C\leq _{T-V}D$ (in the sense of $\bar M$) if whenever $\phi(x)$ is an $L_{C}$-formula such that  ${\bar M}\models \phi(d)$ for some $d$ from $D$ then there is $c$ from $C$ such that ${\bar M} \models \phi(c)$.  

\vspace{2mm}
\noindent
Our underlying assumptions are still that $T = T^{eq}$ is $t.t.$ $A\subseteq B$ are definably closed sets, $B\subseteq M_{A}$ and $B$ is normal over $A$ in $M_{A}$. 

\begin{Definition} We say that $B$ is minimal (over $A$ in $M_{A}$) if there is no $C\supseteq A$ with $C\leq _{TV}B$ and $C\neq B$. 
\end{Definition}

\begin{Remark} We know (see Lemma 18.4 of \cite{Poizat-course} for example) that normality of $B$ over $A$ (in $M_{A}$) implies that $M_{A}$ is also prime over $B$.  The following gives an extension.

\end{Remark}

\begin{Lemma} Suppose $B$ is (normal and) minimal over $A$. Then for any  $C$ with $A\subseteq C \subseteq B$, $M_{A}$ is also prime over $C$. 
\end{Lemma} 
\begin{proof} Let $M_{C}$ be a prime model over $C$ contained in $M_{A}$.
\newline
{\em Claim.} $B\subseteq M_{C}$.
\newline
{\em Proof of claim.}  Consider $D = M_{C}\cap B$.  Let $d$ be a (finite) tuple from $D$ and $b'$ a (finite) tuple from $B$ and $\phi$ a formula over $A$ say such that  $\models \phi(d,b')$. Now $tp(b'/A)$ is isolated by a formula $\psi(y)$.  
So the formula $\phi(d,y) \wedge \psi(y)$ is realized in $M_{A}$, so also realized in $M_{C}\prec M_{A}$, by $b''$ say. 
But all realizations of $\psi(y)$ in $M_{A}$ are in $B$ (by normality), so $b''\in M_{C}\cap B$. 
We have shown that $D \leq_{TV} B$. Hence by our assumptions $D = B$ and so $B\subseteq M_{C}$.  \qed

Now $M_{C}$ is clearly atomic over $A$, so prime over $A$. Now (by the claim) $B\subseteq M_{C}$ and is clearly normal over $C$ in $M_{C}$ whereby, using Remark 2.11, $M_{C}$ is also prime over $B$.  So by Remark 2.11 and uniqueness $M_{C}$ is isomorphic to $M_{A}$ over $B$, so also over $C$, and so $M_{A}$ is prime over $C$.
\end{proof}

\begin{Lemma} Assume $B$ (normal and) minimal in $M_{A}$. Let $C$ be definably closed with $A\subseteq C \subseteq B$. Let $H = \{\sigma\in Aut(B/A): \sigma$ fixes $C$ pointwise\}. Then $C$ is precisely the set of elements of $B$ fixed by $H$.
\end{Lemma} 
\begin{proof}  Suppose that $b\in B\setminus C$. As $M_{A}$ is prime over $C$, $tp(b/C)$ is isolated. As $b\notin dcl(C)$, $tp(b/C)$ has another realization $b'$ in $M_{A}$ which will be in $B$ by normality., As $M_{A}$ is $\omega$-homogeneous over $C$ there is $\sigma\in Aut(M_{A}/C)$ with $\sigma(b) =b'$. Then $\sigma|B \in Aut(B/A)$ and moves $b$. 

\end{proof} 

Hence, bearing in mind Proposition 2.7 (ii), we have:
\begin{Theorem}\label{thm: correspondence} Suppose $A\subseteq B \subseteq M_{A}$ with $B$ definably closed and both normal and minimal (over $A$ in $M_{A}$). Then there is a Galois correspondence between definably closed subsets of $B$ containing $A$ and closed subgroups $H$ of ${\mathcal G} = Aut(B/A)$. Namely $C$ corresponds to 
$H_{C} = \{\sigma\in {\mathcal G}: \sigma$ fixes pointwise $C \}$,  $H$ corresponds to $C_{H}$ the set of elements of $B$ fixed by $H$, and $C_{H_{C}} = C$ and $H_{C_{H}} = H$. 
\end{Theorem}

Finally in this section we make the connection with works of Poizat referred to above. In the context of $T$ totally transcendental, Poizat defines $A_{min}$,  the {\em minimal closure}  of $A$ to be the intersection of all $f(M_{A})$ as $f$ ranges over $A$-elementary embeddings of $M_{A}$ into itself. The notation $A_{min}$ comes from \cite{Poizat-course}.  In \cite{Poizat-theories-stables} Poizat uses the notation $Cl(A)$ for $A_{min}$ where an equivalent definition of $A_{min}$ is given, in terms of {\em atomic algebraicity}.  

\begin{Lemma} 

\begin{enumerate}
    \item $A_{min}$ is the greatest (under inclusion) $B$ with $A\subseteq B \subseteq M_{A}$ such that $B$ is normal and minimal (over $A$ in $M_{A}$).
    \item Any subset $B$ of $A_{min}$ is minimal (namely there is no proper subset $C$ of $B$ containing $A$ with $C\leq_{TV}B$).  
\end{enumerate}
\end{Lemma}  
\begin{proof} 
\begin{enumerate}
    \item Corollaire 7 of \cite{Poizat-theories-stables} (actually there are two Corollaire 7's) say that $f(A_{min}) = A_{min}$ for every $A$-automomorphism of $M_{A}$, so we have normality.  Corollaire 5 of \cite{Poizat-theories-stables} says in particular that $A_{min}$ is atomic over any subset of itself which contains $A$, so if $B\supseteq A$ is a proper subset of $A_{min}$, we could not have that $B\leq_{TV}A_{min}$.
    
    Now suppose that $B\supseteq A$ is normal and minimal (over $A$ in $M_{A}$).  Let $f$ be an $A$-elementary embedding of $M_{A}$ into itself.
    
    As $B$ is normal, $f(B) \subseteq B$ and we clearly have $f(B)\leq_{TV} B$. By minimality of $B$, $f(B) = B$. Hence $B\subseteq f(M)$, so $B\subseteq A_{min}$. 
    \item  By part 1 and Lemma 2.11 (or Corollaire 5 of \cite{Poizat-theories-stables}), $M_{A}$ is prime over any subset $C$ of itself which contains $A$. The desired result follows.

\end{enumerate}
\end{proof}

So bearing in mind the last lemma, Poizat's results, especially from \cite{Poizat-theories-stables} give a lot of information about minimal normal sets.

\section{The Picard-Vessiot closure} 
The ambient theory here is $DCF_{0}$ the theory of differentially closed  fields of characteristic $0$ in the language $L_{\partial}$ of rings with a derivation (so the ring language) together with symbol $\partial$.  We may use just $L$ for the ring language. $DCF_{0}$ is complete, $\omega$-stable, with quantifier elimination in $L_{\partial}$. $DCF_{0}$ also has elimination of imaginaries so $DCF_{0} = DCF_{0}^{eq}$. 

We work in a big saturated model ${\mathcal U}$ of $DCF_{0}$.  $K$ will denote a ``small" differential subfield. We will write $K^{diff}$  (for differential closure of $K$) for the prime model $M_{K}$ over $K$. We assume that $C_{K}$, the field of constants of $K$, is algebraically closed. This assumption has a number of consequences, for example $C_{K^{diff}} = C_{K}$. 

We repeat material from Section 3 of  \cite{Pillay-aut} and Section 1.2 of \cite{Meretzky-Pillay-torsor}. See also \cite{MagidI}, \cite{MagidII}.  However as mentioned in the introduction we will follow the  conventions from \cite{Wibmer-regular-singular}.
A (homogeneous) linear differential equation (in vector form) over $K$ is an equation of the form $\partial Y = AY$ where $Y$ is a $n\times 1$ column vector of indeterminates and $A$ is an $n\times n$ matrix over $K$. The solution set ${\mathcal Y}(K^{diff})$ of the equation in $K^{diff}$ (a collection of $n\times 1$ column vectors from $K^{diff}$)  is an $n$-dimensional vector space  over $C_{K} = C_{K^{diff}}$. The (differential) field $L$ generated over $K$ by the set of all these solutions, equivalently by a $C_{K}$-basis for the solution set, is a called the $PV$ (Picard-Vessiot) extension of $K$ for the given equation.   
Given a family ${\mathcal F}$ of linear  $DE$'s over $K$  (say $\{\partial Y_i = A_iY_i : i\in I\}$) by the $PV$-extension of $K$ for the family ${\mathcal F}$ we mean the differential field extension of $K$ generated by the solutions in $K^{diff}$ of all these equations.  By definition a $PV$ extension $L$ of $K$ is a $PV$  extension for some such family ${\mathcal F}$. We say that $L$ is of {\em finite type} if ${\mathcal F}$ is a single equation (equivalently finitely many equations).  Note that a $PV$ extension $L$ of $K$ is normal and minimal over $K$ (in $K^{diff}$).  Normality is immediate.  For minimality: suppose $K\leq L_{1}\leq_{TV}L$.  Then $L$ will contain a $C_{K}$-basis of the solution set (in $K^{diff}$) of each equation in ${\mathcal F}$. So $L_{1} = L$.

Let us note in passing that the first order  vector form formalism for linear differential equations subsumes and is equivalent to the formalism of higher order homogeneous linear differential equations of one variable.   An order $n$ such equation over $K$ is of the form $y^{(n)} + a_{n-1}y^{(n-1)} + ... + a_{1}y' + a_{0}y$ where the $a_{i}\in K$.  $K^{PV}$ is generated by solutions in $K^{diff}$ of such equations. Moreover given such an equation, one can form an equation in vector form over $K$, $\partial Y = AY$ where $Y$ is an $n\times 1$ column vector of indeterminates such that the solutions of the latter are precisely of the form $(y,\partial(y),...,\partial^{(n-1)}(y))^{t}$ for $y$ a solution of the order $n$ equation.

When  $L$ is a $PV$ extension of $K$ for a single  equation $\partial Y = AY$  then $Aut(L/K)$  (as differential fields) has the structure of an algebraic subgroup of $GL_{n}(C_{K})$ (where $n$ is the length of the column  vector $Y$).
In general $Aut(L/K)$ is an inverse limit of such groups, so a proalgebraic group with respect to the algebraically closed field $C_{K}$. 

By $K^{PV}$  we mean the $PV$-extension of $K$ for the family of {\em all} linear $DE$'s over $K$.  $K^{PV}$ is clearly the greatest $PV$ extension of $K$. (Magid writes $K_{1}$ for $K^{PV}$.)

As $K^{PV}$ is normal over $K$ in $K^{diff}$ we have, by Remark 2.10 that $K^{diff}$ is also a differential closure of $K^{PV}$.  So we can continue the construction to obtain $K^{PV_{2}} = (K^{PV})^{PV}$ the diffferential field generated by all solutions in $K^{diff}$ of all linear differential equations over $K^{PV}$.  Then $K^{PV_{2}}$ is also normal (and minimal) over $K$ in $K^{diff}$, whereas $K^{diff}$ is also a differential closure of $K^{PV_{2}}$.  We continue this was to build $K^{PV_{n}}$ inside $K^{diff}$ for all $n$. We let $K^{PV_{\infty}} = \cup_{n}K^{PV_{n}}$.  To summarise:

\begin{Lemma} (i) Each of $K^{PV_{n}}$ and $K^{PV_{\infty}}$ is normal and minimal over $K$ inside $K^{diff}$ (so by Lemma 2.14 contained in $K_{min}$). 
\newline
(ii)  $K^{PV_{\infty}}$ is $PV$-closed (has no proper $PV$ extension) and is in the fact  $PV$-closure of $K$, smallest $PV$-closed differentrial subfield of $K^{diff}$ containing $K$. 
\end{Lemma} 

\vspace{2mm}
\noindent
 Magid  \cite{MagidII} denotes $K^{PV_{\infty}}$ by $K_{\infty}$ and calls it the complete $PV$-closure of $K$  (although of course ``$PV$-closure" is correct). 

In any case Proposition 2.13 applies to each of $K^{PV_{n}}$ and $K^{PV_{\infty}}$ to give a full Galois correspondence between closed subgroups of the automorphism group over $K$ and intermediate
differential subfields.

\vspace{5mm}
\noindent
Finally let us discuss the structure of differential subfields $L$  in between $K$ and $K^{PV_{\infty}}$ which are normal over $K$ in $K^{diff}$.  The main point is that any such $L$ is an ``iterated" Picard-Vessiot extension of $K$. 

\begin{Proposition} Let $L$ be a differential field such that $K\leq L \leq K^{PV_{\infty}}$ which is normal (over $K$ in $K^{diff}$).  Let $L_{n} = L\cap K^{PV_{n}}$  (and take $L_{0} = K$). 
Then
\newline
(i)  Let $N_{L} \ =\{n<\omega: L_{n+1}\setminus L_{n}\neq\emptyset\}$.  Then $N_{L}$ is an initial segment of $\omega$ (so all of $\omega$ or a finite proper initial segment). 
\newline
(ii)  Each $L_{n+1}$ is a Picard-Vessiot extension of $L_{n}$ (maybe trivial if $L_{n+1} = L_{n}$) and is normal and minimal over $K$ in $K^{diff}$
\newline
(iii) $L$ is also minimal (over $K$ in $K^{diff}$). 
\end{Proposition}
\begin{proof} For simplicity we will consider just the case of $L_{1}$ and $L_{2}$. The reader can easily generalize the proof.  Note that $L_{1}$ is a $PV$ extension (maybe trivial) of $L_{0} = K$.  Suppose that $L_{2}\setminus L_{1}$ is nonempty.  So, by our earlier discussion, there is $b\in L_{2}\setminus L_{1}$ which is a solution of an equation $y^{(n)} + a_{n-1}y^{(n-1)} + ... + a_{1}y' + a_{0}y$ with the $a_{i}\in L_{1}$ and $n > 1$.  We immediately replace $b$ by
${\bar b} =  (b,\partial(b),..,\partial^{(n-1)}(b))^{t}$ (which note is in $L_{2}$) which is a solution of an appropriate $\partial Y = AY$ over $K_{1}$.   Consider $tp({\bar b}/K^{PV})$ which is isolated by a formula $\phi({\bar x}, c)$ where 
$c$ is a tuple from $K^{PV}$ which we may assume to be a ``canonical parameter" of the formula $\phi({\bar x}, c)$. 
Now as $K^{PV}$ is normal in $K^{diff}$, $c$ is fixed by any automorphism of $K^{diff}$ which fixes ${\bar b}$.  By homogeneity $c\in dcl({\bar b})$ in particular $c\in L$ so $c\in L_{1}$.  Note that by normality of $L$ all realizations of 
$\phi({\bar x}, c)$ are in $L_{2}$, and are moreover contained in  $V$, the  set of solutions of $\partial Y = AY$ in $K^{diff}$  (contained in $L_{2}$).  $V$ is an $n$-dimensional vector space over $C_{K}$, and the set $X$ of realizations 
of $\phi({\bar x}, c)$ is an infinite set of realizations of a complete type over $K^{PV}$).  Then, using the Zilber indecomposability theorem for example, the $C_{K}$-vector subspace $W$ of $V$ generated by $X\cup\{0\}$ is definable, over $C_{K},c$, so over $L_{1}$.   Now $W$ is contained in $L_{2}$, and contains the original chosen element $b$. Also on very general grounds the differential field generated over $L_{1}$ by $W$ is a (finite type) $PV$ extension of $L_{1}$.
We have shown that $L_{2} $ is a $PV$ extension of $L_{1}$. 
\end{proof}

So by the Proposition above we can use Proposition 2.13 to strengthen Theorem  5 of \cite{MagidII} by both  replacing ``normal locally iterated $PV$ extension" by just ``normal" and including a definition of ``closed subgroup". 

\begin{Corollary} Let $L$ be a differential subfield of $K^{PV_{\infty}}$ containing $K$ and suppose $L$ is normal (over $K$ in $K^{diff}$). Then there is a full Galois correspondence between intermediate differential fields (between $K$ and $L$) and closed subgroups of $Aut(L/K)$. 
\end{Corollary}

\section{$Aut(K_{2}/K)$}
Our aim to  prove Proposition \ref{prop: proalg} in the introduction.  So $K$ is again a differential field with algebraically closed field $C_{K}$ of constants. 
Using the notation now of Section 3, we take $L$ to be a normal differential subfield of $K^{PV_{\infty}}$ such that $L = L_{2}$ and $L_{1}\neq L_{2}$.  Equivalently, $L$ is a normal differential subfield of  $K^{PV_{2}}$ which is not contained in $K^{PV}$, equivalently, $L$ is not a $PV$-extension of $K$.  By Proposition 3.2, $L_{1}$ is a (nontrivial) $PV$ extension of $K$ and $L_{2}$ is a $PV$ extension of $L_{1}$. (Moreover there is no differential subfield $F$ of $L_{2}$ properly containing 
$L_{1}$ which is a $PV$ extension of $K$.)

A special case is, of course, taking $L_{1} = K^{PV}$ and $L_{2} = K^{PV_{2}}$ which is the case appearing in the abstract. 

We will now change notation to be consistent with the introduction. So we will write $K_{1}$ for $L_{1}$ and $K_{2}$ for $L_{2}$ and reserve the symbol $L$ for other things.

We let $G = Aut(K_{2}/K)$, $N = Aut(K_{2}/K_{1})$ and $H = Aut(K_{1}/K)$.  So we have the exact sequence
$$ 1 \to N \xrightarrow{i} G \xrightarrow{\pi} H \to 1$$
of abstract groups for now, where $i$ is the natural embedding.

As discussed in \cite{Pillay-aut}, $H$ has the structure of a pro-affine algebraic group over $C_{K}$, and likewise for $N$. 
Let us be more precise, focusing on $N = Aut(K_{2}/K_{1})$.
For each finite-type Picard-Vessiot extension $L$ of $K_{1}$ contained in $K_{2}$, picking a fundamental system of solutions ${\bar b}$ in $L$ for a suitable linear $DE$ over $K_{1}$, yields a group isomorphism $g_{L}$ of $Aut(L/K_{1})$ with a linear algebraic subgroup $G_{L}$ of $GL_{n}(C_{K})$, where $g_{L}(\sigma)$ is the  unique $n\times n$ nonsingular matrix $c_{\sigma}$ over $C_{K}$ such that $\sigma({\bar b}) = {\bar b}\cdot c_{\sigma}$. So in fact $G_{L} =\{{\bar b}^{-1}\sigma({\bar b}):\sigma\in Aut(L/K_{1})\}$.  Here multiplication is in the sense of $GL_{n}(K^{diff})$. 
Given $L\leq L'$ where $L'$ is a larger finite-type $PV$ extension of $K_{1}$ inside $K_{2}$ we have a surjective homomorphism from $Aut(L'/K_{1})$ to $Aut(L/K_{1})$ inducing via $g_{L}$ and $g_{L'}$ a surjective homomorphism  $f_{L',L}$ of algebraic groups from $G_{L'}$ to $G_{L}$.  Then by construction $\{(G_{L},f_{L,L'})\}_{L}$ is an inverse system of linear algebraic groups over $C_{K}$, and we obtain an isomorphism $h$ between $N$ and $\varprojlim G_{L}$ a pro-affine algebraic group over $C_{K}$ (or more accurately its group of $C_{K}$-points). 

It is important to discuss the ``intrinsic" manifestation of the Galois group as it will play a role in our  computations.   Consider as above a finite type $PV$ extension $L$ of $K_{1}$ inside $K_{2}$ and ${\bar b}$ a fundamental system of 
solutions in $L$ for a suitable linear $DE$ over $K_{1}$.  For $\sigma\in Aut(L/K_{1})$ let $h_{L}(\sigma) = \sigma({\bar b}){\bar b}^{-1}\in GL_{n}(K^{diff})$. Then the definition of $h_{L}(\sigma)$ does not depend on the choice ${\bar b}$ of the fundamental system of solutions of the given linear $DE$ and $h_{L}$ sets up an isomorphism between $Aut(L/K_{1})$ and a $K_{1}$-definable (differential algebraic) subgroup $H_{L}$ of $GL_{n}(K^{diff})$. 
 If $X$ is the set of realizations of $tp({\bar  b}/K_{1})$ (in $K^{diff})$, then $H_{L}$ acts (by multiplication in $GL_{n}$) on the left on $X$, isomorphic to the action of $Aut(L/K_{1})$ on $X$  (and moreover the action is strictly transitive). 
$H_{L}$ and $G_{L}$ are definably isomorphic over the parameter ${\bar b}$ (and their own differential fields of definition).  Again if $L'\geq L$ is another finite type $PV$ extension of $K_{1}$ contained in $L_{2}$, we obtain the 
differential algebraic group $H_{L'}$ and the natural surjective homomorphism $Aut(L'/K_{1}) \to Aut(L/K_{1})$ gives a surjective homomorphism from $H_{L'}$ to $H_{L}$ defined over $K_{1}$.  We obtain again an inverse limit 
$\varprojlim _{L}H_{L}$ which is a pro-differential algebraic group, defined over $K_{1}$, and definably isomorphic to  $\varprojlim G_{L}$.  In the statement of Proposition 1.4) for example we are identifying $N$  with $\varprojlim G_{L}$ above. But it will be useful to compute with the pro-differential algebraic group $\varprojlim H_{L}$ too, and we will do this in our proofs below. 

\vspace{2mm}
\noindent


Let us make a few comments on so-called $*$-definabillity in model theory.  Such a $*$-definable set refers to the set of solutions of a partial type in maybe infinitely many variables $(x_{i})_{i\in I}$. To ``see" all of the points one usually will work in a saturated model $M$ such that $|I| < |M|$ and the partial type is over a ``small" set of parameters.  By a $*$-definable group we mean a $*$-definable set together with a $*$-definable group operation. A special case of a $*$-definable group is an inverse limit of definable groups with the definable connecting homomorphisms being surjective (and with a small directed system with respect to the saturation of the ambient model). In the totally transcendental context every $*$-definable group is of this form.   However for the example of $\varprojlim G_{L}$, we are computing the group in the model $K^{diff}$ which is not saturated and, in any case, the inverse system could be very large. We happen to see ``all" of the points in $K^{diff}$ because of the isomorphism between the inverse limit and the automorphism group $Aut(K_{2}/K_{1})$.

Let us also remark that the group $G = Aut(K_{2}/K)$ above could not in any reasonable sense have the structure of an inverse limit of definable groups. In that case it would be a proalgebraic group over $C_{K}$ which is a extension of an pro-affine algebraic group by a pro-affine algebraic group, so would itself be pro-affine.  But then $K_{2}$ would itself be a $PV$ extension of $K$, which it is not. 
In any case we would like to recover some additional (pro-)definable or proalgebraic structure on $G$. This is Proposition 1.4.


Let us fix $D$ a left/right coset of $N$ in $G$, So $D = \pi^{-1}({\sigma_{0}})$ for some $\sigma_{0}\in H$.  Namely $D = \{\tau\in G = Aut(K_{2}/K): \tau|K_{1} = \sigma_{0}\}$.  We will take $\sigma_{0}\neq id$. 
Let us fix a $PV$ extension $L$ of $K_{1}$ contained in $K_{2}$. Then for $\tau\in D$, $\tau(L)$ depends only on $\sigma_{0}$ because $L$ is generated by a solution to a linear ODE defined over $K_1$.  So let $L' = \tau(L)$ for some/any $\tau\in D$.
Let $D_{L} = \{\tau|L: \tau\in D\}$.  Note that as $D$ is a left coset $\tau N$ of $N$, we have a right action of $N$ on $D$ (multiplication in $G$), and likewise we have a left action of  $N$ on $D$. 

\begin{Remark} (i) Restriction to $L$ gives a right (strictly transitive) action of $Aut(L/K_{1})$ on $D_{L}$.  
\newline
(ii) Restriction to $L'$ gives a left (strictly transitive) action of $Aut(L'/K_{1})$ on $D_{L}$. 
\newline 
(iii) So we  have a bi-torsor  $(Aut(L'/K_{1}), D_{L}, Aut(L/K_{1}))$. 
\end{Remark} 
\begin{proof} Immediate. The bi-torsorial nature (left and right actions commute) is because everything is induced by multiplication in $G$ (which is associative). 

\end{proof} 

\begin{Proposition} The bitorsor in (iii) above is $K_{1}$-definable. More specifically there is a bijection $h: D_{L} \to C_{L}$ for some $K_{1}$-definable $C_{L}$,  and $K_{1}$-definable left action of $H_{L'}$ on $C_{L}$,  and $K_{1}$-definable right action of $H_{L}$ on $C_{L}$,  such that $(h_{L'}, h_{L}, h)$ give an isomorphism of bitorsors  between \newline $(Aut(L'/K_{1}), D_{L}, Aut(L/K_{1}))$ and $(H_{L'}, D_{L},  H_{L})$. 
\end{Proposition}


\vspace{2mm}
\noindent
We will make use of the analysis from Sections 3 and 4 of \cite{Kamensky-Pillay} which we now  summarize to fix notation (before beginning the proof of Proposition 4.2).  

Let us suppose that our definitions of $G_{L}$ and $H_{L}$ earlier in this section depended upon choosing a linear $DE$ $\partial Y = AY$ and a fundamental matrix of solutions $b$ such that $L = K(b)$.  So $b$ is a single solution of the logarithmic 
differential equation $\partial Z = AZ$ where $Z$ is a matrix of indeterminates ranging over $GL_{n}$.  Now let ${\mathcal Y}_{A}$ denote the solution set of $\partial Z = AZ$ in $K^{diff}$ (so ${\mathcal Y}_{A}\subseteq GL_{n}(L)$).  
As defined earlier, $h_{L}$ denotes the map from $Aut(L/K_{1})$ to $GL_{n}(K^{diff})$ taking $\sigma$ to $\sigma(b)b^{-1}$ (multiplication and inversion in the sense of $GL_{n}$).  Note that $\sigma\in Aut(L/K_{1})$ is determined by its action on  ${\mathcal Y}_{A}$, 
As in Lemma 3.3 of \cite{Kamensky-Pillay} $\rho$ does not depend on the choice of $b\in {\mathcal Y}_{A}$,  and moreover establishes the isomorphism $h_{L}$ between $Aut(L/K_{1})$ acting on ${\mathcal Y}_{A}$ and a differential algebraic 
subgroup $H_{A}^{+}$ of $GL_{n}(K^{diff})$ acting on ${\mathcal Y}_{A}$ by left multiplication in $GL_{n}(K^{diff})$.   $H_{A}^{+}$ is what we called $H_{L}$ earlier but we use the new notation as it is consistent with 
\cite{Kamensky-Pillay} and also emphazises the dependence  on $A$.  $H_{A}^{+}$ is the ``intrinsic" differential Galois group of $L$ over $K_{1}$. 

Now ${\mathcal Y}_{A}$ is a union of $H_{A}^{+}$ orbits, equivalently (in $GL_{n}(K^{diff})$)  a union of right cosets of $H_{A}^{+}$. For some $m$ there is an $A$-definable subset $O_{A}$ of $(C_{K})^{m}$ and an $A$-definable surjective function $f_{A}:{\mathcal Y}_{A}\to O_{A}$  inducing a bijection between ${\mathcal  Y}_{A}/H_{A}^{+}$ and $O_{A}$.  By construction the complete types over $K_{1}$ of fundamental matrices of solutions for $\partial Y = AY$   correspond to the fibres $f_{A}^{-1}(a')$ for $a'\in O_{A}$. Let $a = f_{A}(b)$ where $b$ was our above choice of fundamental matrix of solutions.   The algebraic subgroup $G_{L}$ of $GL_{n}(C_{K})$ constructed earlier depends only on $A$ and  $a$ and is precisely $\{b_{2}^{-1}\cdot b_{1}: b_{1}, b_{2}\in ({\mathcal Y}_{A})_{a}$\}.  In \cite{Kamensky-Pillay} it is denoted $(H_{A})_{a}$. 


Let $\sigma_{0}(A) = B$. So $L'$ is the $PV$ extension of $K_{1}$ for the linear $DE$  $\partial Y = BY$.   We can define  ${\mathcal  Y}_{B}$, $(H_{B})^{+} = H_{L'}$,  $O_{B}$ and $f_{B}: {\mathcal Y}_{B} \to O_{B}$ as before. 

Just for the record:
\begin{Remark} For any $\tau\in D_{L}$ we have:
\newline 
(i)  $\tau({\mathcal Y}_{A}) = {\mathcal Y}_{B}$.
\newline
(ii) $\tau(f_{A}) = f_{B}$.
\newline
(iii) $O_{A} = O_{B}$ and is fixed pointwise by $\tau$.
\newline
(iv) $f_{B}(\tau(b)) = a = f_{A}(b)$.
\newline
(v)  $\tau$ is determined by $\tau(b)$.
\newline
(vi)  $({\mathcal Y}_{B})_{a}$ is precisely the set of $\tau'(b)$ as $\tau'$ ranges over $D_{L}$
\end{Remark}
{\em Proof.}  (i) to (iv) are clear. For (v), every element $b_{1}$ of ${\mathcal Y}_{A}$ is of the form  $b\cdot c$ for some $c\in GL_{n}(C_{K})$ so $\tau(b_{1}) = \tau(b\cdot c) = \tau(b)\cdot c$. 
\newline
(vi):   Notice that $tp(b/K_{1})$ is isolated by  ``$y\in ({\mathcal Y}_{A})_{a}$", hence $\sigma_{0}(tp(b/K_{1}))$ is isolated by ``$y\in ({\mathcal Y}_{B})_{a}$" so $\tau(b)$
in  $({\mathcal Y}_{B})_{a}$ and for any $d\in ({\mathcal Y}_{B})_{a}$ there is $\tau'\in D_{L}$ such that $\tau'(b) = d$. 
\qed

\vspace{5mm}
\noindent
{\em Proof of Proposition 4.2.}  Recall our fixed $b\in {\mathcal Y}_{B}$.  Let $C_{L} = \{\tau(b)\cdot b^{-1}:\tau\in D_{L}\}$. By  Remark 4.3 (v) the map $h$ taking  $\tau\in D_{L}$ to $\tau(b)\cdot b^{-1}$ is a bijection between $D_{L}$ and $C_{L}$. 
\newline
{\em Claim 1.}  $C_{L}$ and $h: D_{L} \to C_{L}$ do not depend on the choice of $b\in {\mathcal Y}_{A}$.
\newline
{\em Proof of claim.}  Let $b_{1}\in {\mathcal Y}_{A}$. So $b_{1} = b\cdot c$ for some $c\in GL_{n}(C_{K})$. So for any $\tau\in D_{L}$,
\newline
 $\tau(b_{1})\cdot b_{1}^{-1} = \tau(b\cdot c)\cdot (b\cdot c)^{-1} = \tau(b)\cdot c\cdot c^{-1}\cdot b^{-1}$.  \qed

Recall again that the isomorphism $h_{L}: Aut(L/K_{1}) \to H_{L} = (H_{A})^{+}$ taking $\sigma$ to $\sigma(b)\cdot b^{-1}$ does notdepend on the choice of $b\in {\mathcal Y}_{A}$. 

\noindent
{\em Claim 2.}  Let $\sigma\in Aut(L/K_{1})$ and  $\tau\in D_{L}$, then $h(\tau\sigma) = h(\tau)\cdot h_{L}(\sigma)$. (where $\tau\sigma$ is multiplication in $G$ and $\cdot$ is multiplication in 
$GL_{n}(K^{diff})$.
\newline
{\em Proof of claim.}  So  fix $\tau$ and $\sigma$. By Claim 1, $h(\tau) = \tau(\sigma(b))\cdot \sigma(b)^{-1}$ as $\sigma(b) \in {\mathcal Y}_{A}$.
\newline
Hence $h(\tau)\cdot h_{L}(\sigma) = \tau(\sigma(b))\cdot\sigma(b)^{-1}\cdot \sigma(b)\cdot b^{-1} = \tau\sigma(b)\cdot b^{-1} = h(\tau\sigma)$. 
\qed

\vspace{2mm}
\noindent
Claim 2 shows the definability (over $K_{1}$) of the right action of $Aut(L/K_{1})$ on $D_{L}$.
A similar proof gives definability of the right action  of $Aut(L'/K_{1})$ on $D_{L}$.   This completes the proof of Proposition 4.2.   \qed.

\begin{Corollary} Fix $\tau\in G$. Then  conjugation by $\tau$ induces a definable isomorphism between the definable groups $H_{A}^{+}$ and $H_{B}^{+}$. 

\end{Corollary} 
\begin{proof} For $\sigma\in Aut(L/K_{1})$, $\tau\sigma\tau^{-1}$ is the unique $\sigma'\in Aut(L'/K_{1}$ such that  $\tau\sigma = \sigma'\tau$. 
Identifying via $h_{L}$, $D_{L}$ with $C_{L}$, and $Aut(L/K_{1})$, $Aut(L'/K_{1})$ with $H_{A}^{+}$, $H_{B}^{+}$ respectively, Proposition 4.2 yields that  maps $H_{A}^{+}$ to $C_{L}$ taking taking $\sigma$ to $\tau\sigma$, and from $H_{B}^{+}$ to $C_{L}$  taking $\sigma'$ to $\sigma'\tau$ are definable. This suffices. 
\end{proof}

\begin{Corollary}  (i) For $\tau\in G$, conjugation by $\tau$ is a prodefinable  automorphism of the prodefinable group $\varprojlim _{L}H_{L}$,
\newline
(ii) For $\tau\in G$ conjugation by $\tau$ is a proalgebraic automorphism of the proalgebraic group  $\varprojlim G_{L}$
\end{Corollary}
\begin{proof} (i) Fix $\tau$. Then $\varprojlim _{L}H_{L}  =  \varprojlim _{L}H_{\tau(L)}$ So we can apply Corollary 4.4.
\newline
(ii)  Under the isomorphism between $\varprojlim _{L}H_{L}$ and  $\varprojlim G_{L}$, conjugation by $\tau\in Aut(K_{2}/K)$ gives a prodefinable automorphism of 
 $\varprojlim G_{L}$.  By the stable embeddedness of $C_{K}$ in $K^{diff}$, it will be a proalgebraic automorphism.

\end{proof}

\begin{Remark}  (i)  In fact in  Proposition 4.2, and Corollaries 4.3 and 4.4 we do not make  any use of the assumption that $K_{1}$ is a $PV$-extension. We use nothing about the Picard-Vessiot extension $K_{1}$ of $K$.   
So the content is as follows: Suppose $C_{K}$ is algebraically closed. Let $K\leq K_{1}\leq K_{2}\leq K^{diff}$ be such that $K_{1}$ and $K_{2}$ are normal over $K$ in $K^{diff}$ and  $K_{2}$ is a Picard-Vessiot extension of $K_{1}$.  The conjugation by any given element of $G = Aut(K_{2}/K)$ is a proalgebraic automorphism of $N = Aut(K_{2}/K_{2})$ where the latter is considered as the group of $C_{K}$-points of a proalgebraic group over $C_{K}$. 
\newline
(ii) However the next Corollary does make use of the Picard-Vessiot extension $K_{1}$ of $K$.
\end{Remark} 

\begin{Corollary}  Suppose $C_{K}$ is algebraically closed, $K_{1}$ is a $PV$ extension of $K$ and $K_{2}$ a $PV$ extension of $K_{1}$ and $K_{2}$ is normal over $K$ in $K^{diff}$. Suppose moreover that $G = Aut(K_{2}/K)$ is a semidirect product  $N \rtimes H$ of $N = Aut(K_{2}/K_{2})$ and $H = Aut(K_{1}/K)$. Then for any $g\in N$, left translation by $g$ is a proalgebraic morphism in  $N\times H$ viewed as a proalgebraic variety. 
\end{Corollary} 
\begin{proof} So as both $N$ and $H$ have the structure of proalgebraic varieties (over $C_{K}$), so does the Cartesian product $N\times H$.  
Fix $g\in G$ which we write as $(n_{1}, h_{1})\in N\times H$.  Then left multiplication by $g$ in the group $G$ takes $(n_{2}, h_{2})$ to $(n_{1}n_{2}^{h_{1}}, h_{1}h_{2})$. 
As by Corollary 4.5,  conjugation by the fixed element $h_{1}\in G$ is a proalgebraic automorphism of $N$ and both multiplication in $N$ and in $H$ are proalgebraic, it follows that left multiplication by $g$ is a proalgebraic morphism from $N\times H$ to itself.

\end{proof}


\begin{thebibliography}{99}

\bibitem{Kamensky-Pillay} M. Kamensky and A. Pillay, Interpretations and differential Galois extensions, IMRN vol. 2016 (2016), 7369 - 7389.
\bibitem{KPR}
K. Krupi\'nski, A. Pillay, and T. Rzepecki,
Topological dynamics and the complexity of strong types,
Israel Journal of Mathematics, 228 (2018), no.~2, 863--932.
\bibitem{Lascar-Pillay}
D. Lascar and A. Pillay,
Hyperimaginaries and automorphism groups,
Journal of Symbolic Logic, 66 (2001), no.~1, 127--143.
\bibitem{MagidI} A. Magid, The Picard-Vessiot closure in differential Galois theory, in Differential Galois theory, Banach Centre publications, vol 58. , IMPAN, Warsaw 2002. 
\bibitem{MagidII} A. Magid, The complete Picard-Vessiot closure,  arXiv:2203.00705 
\bibitem{Meretzky-Pillay-torsor} D. Meretzky and A. Pillay,  Picard-Vessiot extensions, linear differential algebraic groups, and their torsors, Communications in Algebra, 53 (2025), 148 - 161. 
\bibitem{Meretzky-definable-GC}
D. Meretzky,
The short exact sequence in definable Galois cohomology,
The Journal of Symbolic Logic (2025), 1--15,
doi:10.1017/jsl.2025.8.
\bibitem{Pillay-Galois} A. Pillay, Remarks on Galois cohomology and definability, Journal Symbolic Logic,  62 (1997), 487 - 492. 
\bibitem{Pillay-aut} A. Pillay, Automorphism groups of prime models, and invariant measures, Annals of Pure and Applied Logic., 176 (2025) 103568.
\bibitem{Pillay-Poizat}
A. Pillay and B. Poizat,
Pas d’imaginaires dans l’infini!,
The Journal of Symbolic Logic, 52 (1987), no.~2, 400--403.
\bibitem{Poizat-course}  B. Poizat, A Course in Model Theory, Springer, 2000. 
\bibitem{Poizat-Galois} B. Poizat,  Une theorie de Galois imaginaire,  Journal of Symbolic Logic 48 (1983), 1151 - 1170
\bibitem{Poizat-theories-stables} B. Poizat, Modeles premier d'une theorie totalement transcendanty,  Groupe d'etude de theories stables, tome 2 (1978 - 79),  exp. no. 8, p. 1-7.  (Available online at https://www.numdam.org/actas/STS/)
\bibitem{Wibmer-regular-singular}  M. Wibmer, Regular singular differential equations and free proalgebraic groups, Bulletin London Math. Soc, vol. 56 (2024), 2568-2583. 



\end{thebibliography}
\end{document}